\documentclass[a4paper,10pt,reqno]{amsart}

\theoremstyle{plain}
\newtheorem{theorem}{Theorem}[section]
\newtheorem{maintheorem}{Main Theorem}[section]
\newtheorem{proposition}[theorem]{Proposition}
\newtheorem{lemma}[theorem]{Lemma}

\theoremstyle{definition}

\theoremstyle{remark}
\newtheorem{remark}[theorem]{Remark}

\usepackage{tikz}%
\usepackage{multirow}%
\usepackage{amssymb}%

\newcommand{\Z}{\mathbb{Z}}
\newcommand{\calC}{\mathcal{C}}
\newcommand{\calT}{\mathcal{T}}

\newcommand{\qf}[1]{\langle #1 \rangle}
\newcommand{\Case}[1]{\noindent{\textbf{Case #1}}}
\newcommand{\Step}[1]{\noindent{\textbf{Step #1}}}
\newcommand{\binlattice}[4]{\left(\begin{smallmatrix} #1 & #2\\ #3 & #4 \end{smallmatrix} \right)}

\begin{document}

\title[Tight universality of $m$-gonal forms with minimal criterion sets]{Tight universality of $m$-gonal forms with minimal criterion sets}

\author[B. M. Kim]{Byeong Moon Kim}
\address{Department of Mathematics,  Kangwon National University (Gangneung Campus),
7 Jukheon-gil, Gangneung-si, Gangwon-do, 25457, Korea}
\email{kbmgwnu@kangwon.ac.kr}

\author[J. Y. Kim]{Ji Young Kim}
\address{SNU college, Seoul National University,
1 Gwanak-ro, Gwanak-gu, Seoul, 08826, Korea}
\email{jykim98@snu.ac.kr}

\thanks{The authors were supported by Basic Science Research Program through the National Research Foundation of Korea (NRF), funded by the Ministry of Education (RS-2023-00247457, RS-2020-NR053689).}

\subjclass[2020]{Primary 11D09, 11E25, 11E20}
\keywords{$m$-gonal forms, universal forms, tight universal $m$-gonal forms, minimal criterion set}

\begin{abstract}
For integers $m\geq3$ and $n\geq1$, an $m$-gonal form is called tight $\mathcal{T}(n)$-universal if it represents exactly the positive integers $\mathcal{T}(n)=\{ n, n+1, n+2, \ldots \}$. In this paper, we study the minimal criterion set $\mathrm{CS}(m,n)$ for tight $\mathcal{T}(n)$-universality.
Our main result determines $\mathrm{CS}(m,n)$ for all $n \geq8$, $3 \leq m \leq \left\lfloor \frac{3n+1}{2} \right\rfloor$.
More precisely,
    \[
    \mathrm{CS}(m,n)=
    \begin{cases}
    \{ n, n+1, \ldots, 2n-1 \}, & m=5,\\
    \{ n, n+1, \ldots, 2n \},    & m\neq5.
    \end{cases}
    \]
We also establish the corresponding tight $\mathcal{T}(n)$-universality results and show that the upper bound on $m$ is optimal.

\end{abstract}

\maketitle

\section{Introduction}

For an integer $m \geq 3$, an $m$-gonal number is defined by
    \[
    P_m(x)=\frac{ (m-2)x^2 - (m-4)x }{2} 
    \]
for a nonnegative integer $x$. If we allow $x$ to be any integer, then $P_m(x)$ is called a \emph{generalized $m$-gonal number}.
A polynomial of the form
    \[
    f(x_1, x_2, \ldots, x_\ell) = a_1 P_m(x_1) + a_2 P_m(x_2) + \cdots + a_\ell P_m(x_\ell)
    \]
with nonnegative integer coefficients $a_1, a_2, \ldots, a_\ell$ is called a \emph{sum of $m$-gonal numbers} or an \emph{$m$-gonal form}.
For brevity, we denote such an $m$-gonal form $f$ by $\qf{a_1, a_2, \ldots, a_\ell}_m$.

The study of representations of integers as sums of $m$-gonal numbers has a long and venerable history in number theory.
In 1638, Fermat conjectured that every positive integer $n$ can be represented as a sum of at most $m$ of the $m$-gonal numbers.
This conjecture was proven by Lagrange, Gauss, and Cauchy, completing the proof of Fermat's polygonal number theorem.
A natural generalization of Fermat's polygonal number theorem is to determine all $m$-gonal forms that represent every positive integer, which are called \emph{universal $m$-gonal forms}.
In 1862, Liouville initiated this research by determining all universal ternary triangular forms $\qf{a_1, a_2, a_3}_3$.
In 1917, Ramanujan provided a list of 55 candidate universal quaternary diagonal quadratic forms $\qf{a_1, a_2, a_3, a_4}_4$ \cite{sR-1917}.
In 1927, Dickson confirmed Ramanujan's list with one exception, which was shown to be non-universal \cite{leD-1927}.
The classification of universal quadratic forms culminated in the Conway--Schneeberger Fifteen Theorem  \cite{mB-2000, jhC-2000} and the Bhargava--Hanke 290-Theorem \cite{mB-jH-2011}.
The Fifteen Theorem states that a classical positive definite integral quadratic form is universal if and only if it represents the set $\{1, 2, 3, 5, 6, 7, 10, 14, 15\}$.
Similarly, the 290-Theorem provides a criterion for non-classical integral quadratic forms using a finite set of 29 integers. These theorems establish finite criteria for the universality of integral quadratic forms.

Recently, research has extended these finiteness criteria to generalized $m$-gonal forms, focusing on the concept of tight $\calT(n)$-universality. For a positive integer $n$, an $m$-gonal form $f$ is said to be a \emph{tight $\calT(n)$-universal form} if the set of positive integers represented by $f$ is exactly
    \[
    \calT(n)= \{ z \in \Z^+ \mid z \geq n \} = \{n, n+1, n+2, \ldots \}.
    \]
Just as the universality of an integral quadratic form can be verified using a finite set of integers by the Fifteen Theorem or 290 Theorem, it is essential to specify a \emph{minimal criterion set} (or a Conway--Schneeberger set), denoted by $\mathrm{CS}(m,n)$, to verify the tight $\calT(n)$-universality of an $m$-gonal form.
A minimal tight $\calT(n)$-universality criterion set $\mathrm{CS}(m,n)$ of $m$-gonal forms is a finite subset of $\calT(n)$ such that an $m$-gonal form $f$ is tight $\calT(n)$-universal if and only if $f$ represents no positive integer less than $n$ and represents every integer in $\mathrm{CS}(m,n)$.
Moreover, no proper subset of $\mathrm{CS}(m,n)$ has this property.
While specific cases of $\mathrm{CS}(m,n)$ have been determined, a general characterization for arbitrary $m$ and $n$ remains an open problem.
For the pairs $(m, n)$ indicated by black dots in Figure \ref{fig:CS(m,n)}, the sets $\mathrm{CS}(m,n)$ have been determined in previous works.
Previous classifications relevant to Figure~\ref{fig:CS(m,n)} can be found in \cite{mKim-2022, mKim-bkO-2023} and \cite{jJ-mK-2024}.
For a summary of criterion set results and computational data, see \cite[Table 1]{jJ-mK-2023}.

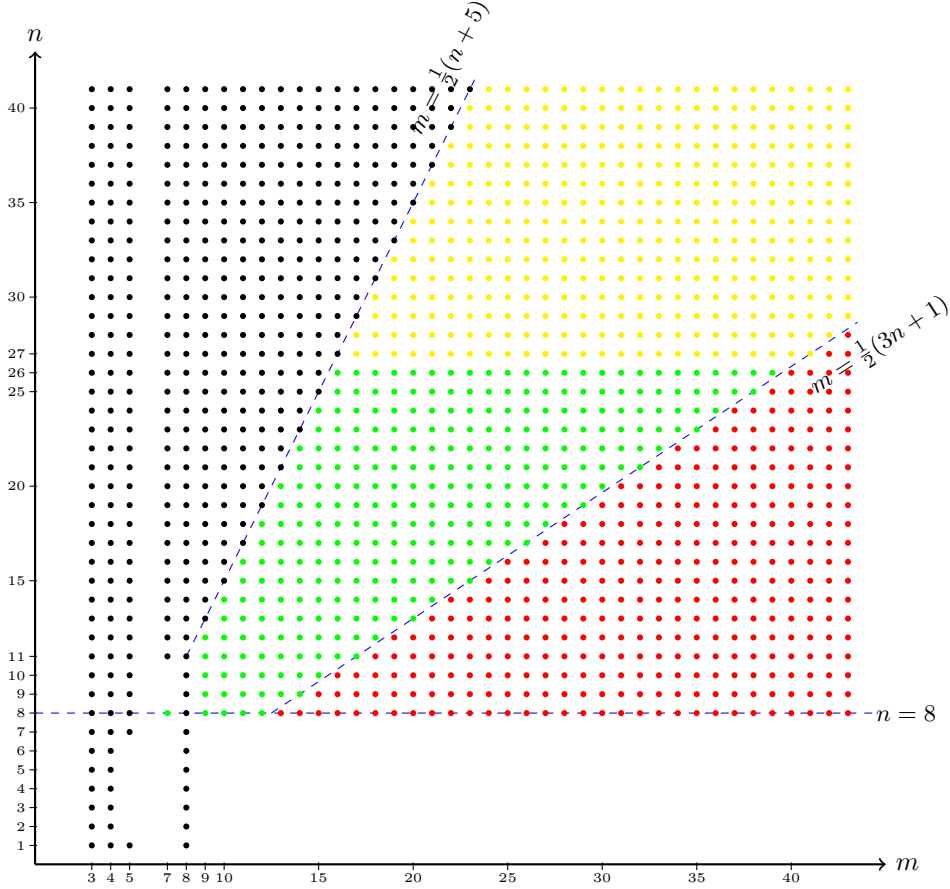
\begin{figure}[htp]
\centering
\begin{tikzpicture}[scale=0.25]

\draw[->, thick] (0,0) -- (45,0) node[right] {$m$};
\draw[->, thick] (0,0) -- (0,43) node[above] {$n$};


\foreach \m in {3, 4, 5, 7, 8, 9, 10, 15, 20, 25, 30, 35, 40} {
  \draw (\m,0.1) -- (\m,-0.3);
  \node[below, font=\tiny] at (\m,0) {\m};
}

\foreach \n in {1, 2, 3, 4, 5, 6, 7, 8, 9, 10, 11, 15, 20, 25, 26, 27, 30, 35, 40} {
  \draw (0.1,\n) -- (-0.3,\n);
  \node[left, font=\tiny] at (0,\n) {\n};
}

\draw[dashed, thin, blue, domain=8:23.25] plot (\x, {2*\x - 5});
\node at (25, 45.5) [above left, font=\small, rotate=63] {$m = \tfrac{1}{2}(n+5)$};

\draw[dashed, thin, blue, domain=12.5:43.5] plot (\x, {(2*\x - 1)/3});
\node at (40, 26) [below right, font=\small, rotate=33] {$m=\tfrac{1}{2}(3n+1)$};

\draw[dashed, thin, blue] (0, 8) -- (44.5, 8);
\node at (44, 8) [right, font=\small] {$n = 8$};


\foreach \m in {3, 4}{
    \foreach \n in {1, ..., 41}{
        \fill[black] (\m,\n) circle (0.16);
    }
}

\foreach \m in {5}{
    \foreach \n in {1, 7, 8, ..., 41}{
        \fill[black] (\m,\n) circle (0.16);
    }
}

\foreach \m in {7}{
    \foreach \n in {11, ..., 41}{
        \fill[black] (\m,\n) circle (0.16);
    }
}

\foreach \m in {8} {
    \foreach \n in {1, 2, ..., 41} {
        \fill[black] (\m,\n) circle (0.16);
    }
}

\foreach \m in {9,...,23} {
    \pgfmathtruncatemacro{\nstart}{2*\m-5}
    \foreach \n in {\nstart,...,41} {
        \fill[black] (\m,\n) circle (0.16);
    }
}

\foreach \m in {9,...,43} {
    \foreach \n in {8,...,41} {
        \pgfmathparse{and(\m > (0.5*\n + 2.5), \m <= (1.5*\n + 0.5))}
        \ifnum\pgfmathresult=1
            \ifnum \n < 27
                \fill[green] (\m,\n) circle (0.16);
            \else
                \fill[yellow] (\m,\n) circle (0.16);
            \fi
        \fi
    }
}

\foreach \n in {8,9,10} {
    \fill[green] (7,\n) circle (0.16);
}

\foreach \m in {13,...,43} {
    \foreach \n in {8,...,41} {
        \pgfmathparse{\m > (1.5*\n + 0.5) ? 1 : 0}
        \ifnum\pgfmathresult=1
            \fill[red] (\m,\n) circle (0.16);
        \fi
    }
}


\end{tikzpicture}
\caption{Known and newly treated pairs for $\mathrm{CS}(m,n)$.
Black dots indicate pairs treated in previous works.
Green, yellow, and red dots indicate the pairs treated in Sections~\ref{Sec:FromT(8)toT(26)}, \ref{Sec:FromT(27)}, and \ref{Sec:Optimality}, respectively.}

\label{fig:CS(m,n)}
\end{figure}

In this paper, we extend the previous results on $\mathrm{CS}(m,n)$ to a much broader range of $(m, n)$ by focusing on two specific families of $m$-gonal forms:
    \[
    F = \langle n, n, n+1, \dots, 2n-1 \rangle_m \quad \text{and} \quad G = \langle n, n+1, \dots, 2n \rangle_m.
    \]
In Section~\ref{Sec:qf{1222}_m}, we provide the necessary preliminaries and establish the representability of sufficiently large integers by the specific $m$-gonal form $\qf{1,2,2,2}_m$, which serves as a building block for our later proofs.
In Section~\ref{Sec:FromT(8)toT(26)}, we prove that the $m$-gonal forms $F$ and $G$ are tight $\calT(n)$-universal for
the pairs $(m,n)$ indicated by the green dots.
These consist of the pairs satisfying
    \[
    8 \leq n \leq 26 \quad \text{and} \quad \tfrac{n+5}{2} < m \leq \left\lfloor \tfrac{3n+1}{2} \right\rfloor,
    \]
together with $(m,n)=(7,9)$ and $(7,10)$.
In Section~\ref{Sec:FromT(27)}, we prove that both forms $F$ and $G$ are tight $\calT(n)$-universal for
    \[
    n \geq 27 \quad \text{and} \quad 3 \leq m \leq \left\lfloor \tfrac{3n+1}{2} \right\rfloor.
    \]
To show that the above universality range is optimal, Section~\ref{Sec:Optimality} focuses on the red region, corresponding to
    \[
    n\geq8  \quad \text{and} \quad  m\geq\tfrac{1}{2}(3n+2).
    \]
We show that, in this region, at least one of the forms $F$ and $G$ fails to represent a certain integer in $\calT(n)$ and hence is not tight $\calT(n)$-universal.
Combining the preceding universality results with the escalation results in \cite[Proposition 3.5]{jJ-mK-2023}, we obtain $\mathrm{CS}(m,n)=\{ n, n+1, \ldots, 2n \}$ for the pairs $(m,n)$ indicated by the green and yellow dots.
Since the set of generalized hexagonal numbers coincides with the set of triangular numbers, the case $m=6$ reduces to the case $m=3$.
Together with the previously known results for the pairs $(m,n)$ indicated by the black dots, these results determine the minimal tight $\calT(n)$-universality criterion set for all pairs $(m,n)$ satisfying $n\geq8$ and $3 \leq m \leq \left\lfloor\frac{1}{2}(3n+1)\right\rfloor$.

\smallskip

\begin{maintheorem}\label{Thm:MainTheorem}
Let $n\geq 8$ and $3 \leq m \le \lfloor \frac{3n+1}{2} \rfloor$. Then
    \[
    \mathrm{CS}(m,n)=
    \begin{cases}
    \{n,n+1,\dots,2n-1\}, & \text{if } m=5,\\[1mm]
    \{n,n+1,\dots,2n\},    & \text{if } m\neq5.
    \end{cases}
    \]
Equivalently, an $m$-gonal form $f$ is tight $\calT(n)$-universal if and only if $f$ represents no positive integer less than $n$ and represents every integer in $\mathrm{CS}(m,n)$.
Moreover, the above range of $m$ is optimal in the following sense.
For every $n\geq8$ and every integer $m \geq \frac{3n+2}{2}$, the set $\{ n, n+1, \ldots, 2n\}$ is not a criterion set for tight $\calT(n)$-universality.
\end{maintheorem}

\smallskip

The remainder of this paper is devoted to establishing the results described above.
To study tight $\calT(n)$-universality, we use the language of quadratic spaces and lattices from the theory of quadratic forms, following the conventions of \cite{otO-1973}. For the necessary notions and terminology, we refer the reader to \cite{jJ-mK-2023}, \cite{jJ-mK-2024} and \cite{mKim-2022}.

\smallskip

\section{Representations by the $m$-gonal form $\qf{1,2,2,2}_m$}\label{Sec:qf{1222}_m}

In this section, we establish several technical results concerning the representability of positive integers by the $m$-gonal form $f(\mathbf{x}) = \qf{1,2,2,2}_m$. These results serve as essential building blocks for the proofs in the subsequent sections. To establish these results, we utilize the  properties of the well-known quadratic form $g(\mathbf{x}) = \qf{1,2,2,2}_4 = x_1^2 + 2x_2^2 + 2x_3^2 + 2x_4^2$. As a result, we show that  $f(\mathbf{x}) = \qf{1,2,2,2}_m$ represents all sufficiently large integers for a given $m$. Throughout this section, let $m\geq3$ be fixed.

For clarity, the $m$-gonal form $f(\mathbf{x}) = \qf{1,2,2,2}_m$ is written as:
    \begin{align*}
    f(\mathbf{x})
    = & P_m(x_1) + 2P_m(x_2) + 2P_m(x_3) + 2P_m(x_4)  \\
    = & \tfrac{(m-2)}{2} \Bigl[ (x_1^2 + 2x_2^2 + 2x_3^2 + 2x_4^2) - (x_1 + 2x_2 + 2x_3 + 2x_4) \Bigr] \\
       & + (x_1 + 2x_2 + 2x_3 + 2x_4)\\
    = & \tfrac{(m-2)}{2} \left[ g(\mathbf{x}) - \ell(\mathbf{x}) \right] + \ell(\mathbf{x})
    \end{align*}
where $g(\mathbf{x}) = x_1^2 + 2x_2^2 + 2x_3^2 + 2x_4^2 = \qf{1, 2, 2, 2}_4$ is a quadratic form and $\ell(\mathbf{x}) = x_1 + 2x_2 + 2x_3 + 2x_4$ is a linear form.

Let $\alpha = (m-2)q +r$ be a positive integer that satisfies $0 \leq r \leq m-3$. If the following system of equations
    \begin{equation}\label{eq:1}\tag{$*$}
    \begin{cases}
    ~ \ell(\mathbf{x}) = (m-2)h + r, \\
    ~ g(\mathbf{x}) = 2(q-h) + (m-2)h + r
    \end{cases}
    \end{equation}
has an integral solution $\mathbf{x}$, then $\mathbf{x}$ is an integral solution of the equation
    \[
    f(\mathbf{x}) = (m-2)q +r,
    \]
that is, the integer $\alpha = (m-2)q +r$ is represented by the $m$-gonal form $f(\mathbf{x}) =\qf{1,2,2,2}_m$.
Note the following relation between the binary lattice $\binlattice{7}{\ell(\mathbf{x})}{\ell(\mathbf{x})}{g(\mathbf{x})}$ and $g(\mathbf{x})=\qf{1,2,2,2}_4$:
    \[
    \binlattice{7}{\ell(\mathbf{x})}{\ell(\mathbf{x})}{g(\mathbf{x})}
    =
    \begin{small}\begin{pmatrix} 1 & 1 & 1 & 1 \\ x_1 & x_2 & x_3 & x_4 \end{pmatrix}\end{small}
    \begin{small}\begin{pmatrix} 1 & 0 & 0 & 0 \\ 0 & 2 & 0 & 0 \\ 0 & 0 & 2 & 0 \\ 0 & 0 & 0 & 2 \end{pmatrix}\end{small}
    \begin{small}\begin{pmatrix} 1 & x_1 \\ 1 & x_2 \\ 1 & x_3 \\ 1 & x_4 \end{pmatrix}\end{small}.
    \]
Moreover, the representation
    \[
    7=1\cdot1^2+2\cdot1^2+2\cdot1^2+2\cdot1^2
    \]
by $g(\mathbf{x})=\qf{1,2,2,2}_4$ is unique up to isometry.
Hence, the system~\eqref{eq:1} has an integral solution if and only if the binary lattice
    \[
    A[h] = \binlattice{7}{~(m-2)h+r~}{~(m-2)h+r~}{~2(q-h)+(m-2)h+r~}
    \]
is represented by the quadratic form $g(\mathbf{x})=\qf{1,2,2,2}_4$.
For notational convenience, we write $|A[h]|:=\det A[h]$.
By the proof of \cite[Lemma 3.3]{dPark-2024}, since $g(\mathbf{x})=\qf{1,2,2,2}_4$ has class number one, the binary lattice $A[h]$ is represented by $g$ whenever $A[h]$ is positive definite and $\det A[h] \neq 4^s(16t + 14)$ for all nonnegative integers $s$ and $t$.
Therefore, the existence of such a binary lattice $A[h]$ guarantees that the integer $\alpha = (m-2)q +r$ is represented by the $m$-gonal form $f(\mathbf{x}) = \qf{1,2,2,2}_m$.

\smallskip

\begin{remark}\label{Rmk:01}\rm
The determinant of $A[h]$ is
    \begin{align*}
    |A[h]| &= 14(q-h)+7((m-2)h+r)-((m-2)h+r)^2 \\
             &= 14q+7r-r^2 -14h + (7-2r)(m-2)h - (m-2)^2 h^2,
    \end{align*}
and
    \[
    |A[h]| - |A[0]| = -14h + (m-2)(7-2r)h - (m-2)^2 h^2
    \]
as $|A[0]| = 14q+7r-r^2$.
This explicit formula will be repeatedly used in the proof of \text{Lemma \ref{Lem:01}.}
\end{remark}

\smallskip

\begin{lemma}\label{Lem:01} 
For each $q, r \in \Z$, there exists an integer $h$ such that $-4 \leq h \leq 3$ and $|A[h]| \neq 4^s(16t + 14)$ for all nonnegative integers $s$ and $t$.
\end{lemma}

\begin{proof}
If $|A[0]| \neq 4^s(16t+14)$ for all $s, t \geq 0$, the statement holds with $h = 0$.
Suppose instead that $|A[0]| = 4^s(16t+14)$ for some $s, t \geq 0$.
Then,
    \[
    |A[0]| = 4^s(16t+14)
              \equiv
              \begin{cases} ~14 &\pmod{16}   \text{ if } s = 0, \\
                                     ~0    &\pmod{8}   ~\text{ if } s \geq 1.
              \end{cases}
    \]
Thus, if we can find $h$ such that
    \[
    |A[h]| - |A[0]| \equiv 4 \pmod{8},
    \]
then $|A[h]| \not\equiv 4^s(16t+14) \pmod{8}$ for all $s, t$, and the proof is complete. 
We will proceed with the proof by dividing into cases based on the conditions of $m$, using the above observation or similar results.

\smallskip

\Case{1)} $m$ is odd:
Set $h = -4$. Then,
    \[
    |A[-4]| - |A[0]| \equiv 56 - 4(m-2)(7-2r) - 16(m-2)^2 \equiv 4 \pmod{8}.
    \]

\smallskip

\Case{2)} $m \equiv 2 \pmod{4}$:
Set $h = 2$. Then,
    \[
    |A[2]| - |A[0]| \equiv -28 + 2(m-2)(7-2r) - 4(m-2)^2 \equiv 4 \pmod{8}.
    \]

\smallskip

\Case{3)}  $m \equiv 0 \pmod{4}$:
In this case, $(m-2)(7-2r) \equiv 2$ or $6 \pmod{8}$, and $(m-2)^2 \equiv 4 \pmod{32}$.
We consider two subcases.

\smallskip

Case 3-a) $(m-2)(7-2r) \equiv 6 \pmod{8}$:
    Set $h = 1$. Then,
    \[
    |A[1]| - |A[0]| \equiv -14 + (m-2)(7-2r) - (m-2)^2 \equiv 4 \pmod{8}.
    \]

\smallskip

Case 3-b) $(m-2)(7-2r) \equiv 2 \pmod{8}$:
    For all odd $h$,
    \[
    |A[h]| - |A[0]| = -14h + (m-2)(7h-2hr) - h^2(m-2)^2 \equiv 0 \pmod{8},
    \]
    as $h \equiv \pm 1 \pmod{4}$.
    If $k, \ell$ are odd integers such that $|A[k]| - |A[0]| \equiv |A[\ell]| - |A[0]| \pmod{32}$, then $k - \ell \equiv 0 \pmod{8}$ because
    \begin{align*}
    & (|A[k]| - |A[0]|) - (|A[\ell]| - |A[0]|) \\
    & = (k-\ell)\underbrace{[-14 + (m-2)(7-2r)]}_{\equiv 4 \pmod{8}} - \underbrace{(k^2 - \ell^2)(m-2)^2}_{\equiv 0 \pmod{32}}
    \end{align*}
    is a multiple of $32$.
    Thus, the residues of $|A[h]| - |A[0]| \pmod{32}$ for $h = -3, -1, 1, 3$ are pairwise distinct.
    Consequently, there exist $h_1, h_2 \in \{-3, -1, 1, 3\}$ such that
    \[
    |A[h_1]| - |A[0]| \equiv 8 \pmod{32}, \quad |A[h_2]| - |A[0]| \equiv 16 \pmod{32}.
    \]
    From
    \[
    |A[0]| = 4^s(16t+14)
              \equiv
              \begin{cases} ~14 &\pmod{16}   \text{ if } s = 0, \\
                                     ~56  &\pmod{64}   \text{ if } s = 1, \\
                                     ~0    &\pmod{32}   \text{ if } s \geq 2.
              \end{cases}
    \]
    we deduce:
    \[
    |A[h_1]| \equiv |A[0]| + 8   \equiv 14 + 8   \equiv 6 \pmod{16}  \text{ if } s=0,
    \]
    and
    \[
   |A[h_2]|  \equiv |A[0]| + 16 \equiv \begin{cases}
                                                             56 + 16 \equiv  8 \pmod{32} & \text{ if } s=1,\\ 
                                                             0 + 16  \equiv 16 \pmod{32} & \text{ if } s \geq 2.
                                                              \end{cases}
    \]
In summary, we can choose $h=h_1$ if $s=0$, and $h=h_2$ if $s \geq 1$.

\smallskip

\noindent In either case, $|A[h]| \neq 4^s(16t+14)$ for all nonnegative integers $s$ and $t$.
\end{proof} 

\smallskip

\begin{theorem}\label{Thm:01}
The $m$-gonal form $f(\mathbf{x})=\qf{1,2,2,2}_m$ represents all positive integers $\alpha$ satisfying $\alpha > B(m)$, where $B(m)=\frac{1}{7}(8m^3-34m^2+68m-64)$.
\end{theorem}

\begin{proof}
Write $\alpha = (m-2)q + r$ such that $0 \le  r \le m-3$.
By Lemma \ref{Lem:01}, there is an integer $h$ such that $-4 \le h \le3$ satisfying $|A[h]| \neq 4^s(16t + 14)$ for all nonnegative integers $s$ and $t$.
To conclude that $\alpha$ is represented by $f=\qf{1,2,2,2}_m$, it suffices to show that the binary lattice $A[h]$ is positive definite.

Since $B(m) = (m-2)\left(\frac{1}{7}(8m^2-18m+25)+1\right)$ and $\alpha=(m-2)q+r<(m-2)(q+1)$, the assumption $\alpha>B(m)$ implies
    \[
    q > \frac{1}{7}(8m^2-18m+25).
    \]
Moreover, from $-4 \leq h \leq 3$ and $0 \leq r \leq m-3$, we have
    \[
    -4(m-2) \leq (m-2)h+r < 4(m-2).
    \]
Therefore,
    \begin{align*}
    |A[h]| &= 14(q-h) + 7((m-2)h + r) - ((m-2)h+r)^2 \\
             &\geq 14(q-h) - 28(m-2) - 16(m-2)^2 \\
             &> 14 \cdot \frac{1}{7}(8m^2 -18m + 25) - 28(m-2) - 16(m-2)^2 -14h \\
             &= 42 - 14h \geq 0.
    \end{align*}
Thus $A[h]$ is positive definite. Consequently, $\alpha$ is represented by \text{$f=\qf{1,2,2,2}_m$}.
\end{proof}

\smallskip

\section{Tight $\calT(n)$-universality of $F$ and $G$ for $8 \leq n \leq 26$}\label{Sec:FromT(8)toT(26)}

In this section, we establish the tight $\calT(n)$-universality corresponding to the green dots in Figure~\ref{fig:CS(m,n)}.
More precisely, we prove that the $m$-gonal forms
    \[
    F = \qf{n, n, n+1, \ldots, 2n-1}_m \quad \text{and} \quad G = \qf{n, n+1, \ldots, 2n}_m
    \]
are tight $\calT(n)$-universal for all pairs $(m,n)$ satisfying $8 \leq n \leq26$ and $\frac{n+5}{2} < m \leq \left\lfloor \frac{3n+1}{2} \right\rfloor$, as well as for the two additional pairs $(7,9)$ and $(7,10)$.

\smallskip

\begin{lemma}\label{Lem:02}
Let $p$ be an odd prime, and suppose that $a$ and $b$ are relatively prime to $p$.
Then, for any integer $k \in \Z$, there exist integers $x_1, x_2$ with
    \[
    0 \leq x_1, x_2 \leq p-1 \quad \text{ such that } \quad aP_m(x_1) + bP_m(x_2) \equiv k \pmod{p}.
    \]
\end{lemma}

\begin{proof}
First suppose that $p \nmid m-2$.
It is known that there exist $u_1, u_2 \in \Z$ satisfying
    \[
    au_1^2 + bu_2^2 \equiv 8k(m-2) + (a+b)(m-4)^2 \pmod{p}.
    \]
Since $p \nmid m-2$, there exist integers $x_1, x_2$ such that $0 \leq x_1,  x_2 \leq p-1$  and
    \[
    2(m-2) x_i - (m-4) \equiv u_i \pmod{p} \quad \text{ for } i=1, 2.
    \]
Substituting this into the expression for $P_m(x)$, we obtain:
     \begin{align*}
    & aP_m(x_1)+bP_m(x_2)\\
    &= a \,\frac{(m-2)x_1^2-(m-4)x_1}{2} + b \,\frac{(m-2)x_2^2-(m-4)x_2}{2}\\
    &= a \,\frac{[2(m-2)x_1 - (m-4)]^2}{8(m-2)} + b \,\frac{[2(m-2)x_2 - (m-4)]^2 }{8(m-2)} - (a+b)\,\frac{~(m-4)^2}{8(m-2)} \\
    &\equiv \frac{au_1^2 + bu_2^2 - (a+b)(m-4)^2 }{8(m-2)} \pmod{p} \\
    &\equiv k \pmod{p}.
    \end{align*}

Now suppose that $p \mid m-2$.
Then $P_m(x) \equiv x \pmod p$ as $m-4 \equiv -2 \pmod p$.
Since $p \nmid a$, we may choose $x_1$ with $0 \leq x_1 \leq p-1$ such that $ax_1 \equiv k \pmod p$ and take $x_2=0$.
Then
    \[
    aP_m(x_1)+bP_m(x_2) \equiv k \pmod p.
    \]
\end{proof}

\smallskip

\begin{remark}\label{Rmk:03}\rm
It is known that for all $n \geq 25$, there always exists a prime $p$ between $n$ and $\frac{6}{5}n$,
as established by Jitsuro Nagura \cite{jNagura-1952} in refinement of Ramanujan's proof of Bertrand's postulate.
Furthermore, it can be verified that for all $12 \leq n \leq 24$, there exists a prime $p$ between $n$ and $2n-10$.
By combining these results, we conclude that for all integers $n \geq 9$, there always exists an odd prime $p$ satisfying
    \[
    n+3 \leq p \leq 2(n+3)-10=2n-4
    \]
which meets the required assumption in Theorem \ref{Thm:02}. In particular,
    \[
    n \leq p-3 \quad\text{and}\quad p+3\leq2n-1,
    \]
so the seven consecutive integers $p-3, p-2, \ldots, p+3$ occur among the coefficients of \text{$H = \qf{n,n+1,\ldots,2n-1}_m$.}
\end{remark}

\smallskip

\begin{theorem}\label{Thm:02}
For $n \geq 9$ and a prime $p$ such that $n+3 \leq p \leq 2n-4$, the $m$-gonal form
    \[
    H = \qf{n, n+1, \ldots , 2n-1}_m
    \]
represents all positive integers $\alpha$ with
    \[
    \alpha > p\,B(m) + (4n-3) P_m (p-1).
    \]
\end{theorem}

\begin{proof}
Let $\alpha$ be a positive integer satisfying
    \[
    \alpha >  p\, B(m)  + (4n-3) P_m (p-1).
    \]
Since there always exists an odd prime $p$ such that $n+3 \leq p \leq2n-4$ (see Remark \ref{Rmk:03}), we may rearrange the coefficients of $H$ as
    \[
    H = \qf{p-3, p-2, p-1, p, p+1, p+2, p+3, a_1, a_2, \ldots, a_{n-7}}_m.
    \]
Since $p > n$ and $2p > 2n-1$, the only multiple of $p$ among $n,n+1,\ldots,2n-1$ is $p$ itself and hence $p \nmid a_1 a_2 \cdots a_{n-7}$.
By Lemma \ref{Lem:02}, there exist integers $x_1, x_2$ such that $0 \leq x_1, x_2 \leq p-1$ and
    \[
    \alpha - (a_1 P_m(x_1) + a_2 P_m(x_2)) \equiv 0 \pmod p.
    \]
Since $P_m(x)$ is increasing for nonnegative integers $x$, $P_m(x_i) \leq P_m(p-1)$ and since $a_1 + a_2 \leq (2n-2) + (2n-1) = 4n-3$, it follows that
    \[
    \alpha - (a_1 P_m(x_1) + a_2 P_m(x_2)) > p\,B(m).
    \]
Thus, there exists an integer $\beta > B(m)$ such that
    \[
    \alpha - (a_1 P_m(x_1) + a_2 P_m(x_2)) = p\,\beta.
    \]
By Theorem \ref{Thm:01}, $\beta$ is represented by the $m$-gonal form $\qf{1, 2 ,2, 2}_m$, meaning that
    \[
    \beta = P_m(y_1) + 2P_m(y_2) + 2P_m(y_3) + 2P_m(y_4)
    \]
for some integers $y_1, y_2, y_3, y_4$.
Therefore, we obtain
    \begin{align*}
    \alpha
    = & \, p \,\beta + a_1P_m(x_1) + a_2P_m(x_2)\\
    = & \, p \,[ P_m(y_1) + 2P_m(y_2) + 2P_m(y_3) + 2P_m(y_4) ] + a_1P_m(x_1) + a_2P_m(x_2)\\
    = & \, (p-3)P_m(y_4) + (p-2)P_m(y_3) + (p-1)P_m(y_2) + pP_m(y_1)  \\
       & + (p+1)P_m(y_2) + (p+2)P_m(y_3) + (p+3)P_m(y_4) + a_1P_m(x_1) + a_2P_m(x_2).
    \end{align*}
This shows that $\alpha$ is represented by the $m$-gonal form $H(\mathbf{x}) = \qf{n, n+1, \ldots , 2n-1}_m$.
\end{proof}

\smallskip

\begin{theorem}\label{Thm:03}
For $8 \leq n \leq 26$, the $m$-gonal forms
    \[
    F = \qf{n, n, n+1, \ldots, 2n-1}_m \quad \text{and} \quad G = \qf{n, n+1, n+2, \ldots, 2n}_m
    \]
represent all positive integers $\alpha$ greater than $\calC[m,n]$, where $\calC[m,n]$ \text{is given in Table \ref{Tbl:Bound C[m,n]}.}
    \begin{table}[htbp]
    \centering
    \renewcommand{\arraystretch}{1.2}
    \begin{small}
    \caption{An upper bound $\calC[m,n]$ for $8 \leq n \leq 26$}\label{Tbl:Bound C[m,n]}
    \begin{tabular}{c|l|l}    \hline
                 & \rm{\hspace{4ex} $m$-gonal  forms} & \rm{\hspace{7ex} conditions on $\calC[m,n]$} \\ \hline
    \multirow{2}{*}{$n=8$}
                  & {$\qf{8, 8, 9, \ldots, 15}_m$}   & $11B(m) + 23P_m(10)$\\ 
                  &{$\qf{8, 9, \ldots, 15, 16}_m$} & $11B(m) + 31P_m(10)$\\ \hline
    \multirow{2}{*}{$9 \leq n \leq 26$}
                 & $\qf{n, n, n+1, \ldots, 2n-1}_m$                     & $(2n-4) B(m)  + (4n-3) P_m (2n-5)$ \\
                 & $\qf{n, n+1, \ldots, 2n-1, 2n}_m$                   & $(2n-4) B(m)  + (4n-3) P_m (2n-5)$\\ \hline
    \end{tabular}
    \end{small}
    \end{table}
\end{theorem}

\begin{remark}\label{Rmk:04}
The bounds in Table~\ref{Tbl:Bound C[m,n]} can be written explicitly as
    \begin{small}
    \begin{align*}
    11B(m)+23P_m(10) &= \tfrac{1}{7} \left(88m^3 - 374m^2 + 7993m - 13584 \right),\\
    11B(m)+31P_m(10) &= \tfrac{1}{7} \left(88m^3 - 374m^2 + 10513m - 18064 \right)
    \end{align*}
and, for $9 \leq n \leq26$,
    \begin{align*}
    (2n-4)B(m)+(4n-3)P_m(2n-5) &= \tfrac{16}{7}(n-2)m^3 - \tfrac{68}{7}(n-2)m^2 \\
                                                      &\quad+ \left( \tfrac{136}{7}(n-2) + (4n-3)(2n-5)(n-3) \right)m\\
                                                      &\quad- \left( \tfrac{128}{7}(n-2) + (4n-3)(2n-5)(2n-7) \right).
    \end{align*}
    \end{small}
\end{remark}

\begin{proof}
Let $\alpha > \calC[m,n]$. We consider three cases according to the value of $n$.

\smallskip

\Case{1)} $n=8$ and $F = \qf{8, 8, 9, \ldots, 15}_m$:
By Lemma \ref{Lem:02}, there exist integers $x_0, x_8$ such that $0 \leq x_0, x_8 \leq 10$ and $8P_m(x_0)+15P_m(x_8) \equiv \alpha \pmod{11}$.
Since $\alpha - (8P_m(x_0) + 15P_m(x_8)) \equiv 0 \pmod {11}$ and
    \[
    \alpha - (8P_m(x_0) + 15P_m(x_8)) \geq \alpha - 23 P_m(10) > 11B(m),
    \]
we can write $\alpha - (8P_m(x_0) + 15P_m(x_8)) = 11 \beta$ for some integer $\beta > B(m)$.
By Theorem \ref{Thm:01},  $\beta = P_m(y_1) + 2P_m(y_2) + 2P_m(y_3) + 2P_m(y_4)$ for some integers $y_1, y_2, y_3, y_4$ and hence
    \begin{eqnarray*}
    \alpha&=&11 [ P_m(y_1)+2P_m(y_2)+2P_m(y_3)+2P_m(y_4) ] + 8P_m(x_0) + 15P_m(x_8) \\
              &=&8P_m(x_0) + 8P_m(y_4) + 9P_m(y_3) + 10P_m(y_2) + 11P_m(y_1)\\
              &  &\hspace{20ex} +12P_m(y_2)+13P_m(y_3)+14P_m(y_4)+15P_m(x_8),
    \end{eqnarray*}
showing that $\alpha$ is  represented by $F = \qf{8, 8, 9, \ldots,15}_m$.

\smallskip

\Case{2)} $n=8$ and $G = \qf{8, 9, \ldots, 15, 16}_m$:
Similarly, Lemma \ref{Lem:02} ensures the existence of integers $x_7, x_8$ such that $0 \leq x_7, x_8 \leq 10$ and $15P_m(x_7)+16P_m(x_8) \equiv \alpha \pmod{11}$.
Since $\alpha - (15P_m(x_7) + 16P_m(x_8)) \equiv 0 \pmod {11}$ and
    \[
    \alpha - (15P_m(x_7) + 16P_m(x_8)) \geq \alpha - 31 P_m(10) > 11B(m),
    \]
we apply the same decomposition as in Case 1) to conclude that $\alpha$ is  represented by \text{$G = \qf{8, 9, \ldots,15, 16}_m$.}

\smallskip

\Case{3)} $9 \leq n \leq 26$ and  $F = \qf{n, n, n+1, \ldots, 2n-1}_m$ and $G = \qf{n, n+1, \ldots, 2n-1, 2n}_m$:
From Remark \ref{Rmk:03}, there always exists an odd prime $p$ with $n+3 \leq p \leq 2n-4$ for all $n \geq 9$.
Since $P_m(x)$ is increasing for nonnegative integers $x$, we have
    \begin{align*}
    \calC[m,n]  & =  (2n-4) B(m)  + (4n-3) P_m (2n-5) \\
                       & \geq  p\,B(m) + (4n-3) P_m (p-1).
     \end{align*}
Hence Theorem~\ref{Thm:02} shows that $H=\qf{n,n+1,\ldots,2n-1}_m$ represents every integer $\alpha > \calC[m,n]$.
Therefore we conclude  that  both $F$ and $G$ represent every positive integer $\alpha > \calC[m,n]$.
\end{proof}

\smallskip

A direct computation using Mathematica verifies that both $F$ and $G$ represent every integer $\alpha$ satisfying
    \[
    n\leq\alpha\leq\calC[m,n]
    \]
for all pairs $(m,n)$ with $8 \leq n \leq 26$ and $\frac{n+5}{2} <m \leq \left\lfloor \frac{3n+1}{2} \right\rfloor$ and also for $(m,n)=(7,9)$ and $(7,10)$.
Since every coefficient of $F$ and $G$ is at least $n$ and $P_m(x) \geq 0$ for every $x \in \Z$, neither form represents a positive integer less than $n$.
Combining these observations with Theorem~\ref{Thm:03}, we obtain the following theorem.

\smallskip

\begin{theorem}\label{Thm:04}%
Let $8\leq n\leq26$. Suppose that either $\frac{n+5}{2} < m \leq \left\lfloor\frac{3n+1}{2}\right\rfloor$ or $(m,n)\in\{ (7,9), (7,10) \}$.
Then the $m$-gonal forms
    \[
    F=\qf{n, n, n+1, \ldots, 2n-1}_m \text{ and } G=\qf{n, n+1, n+2,  \ldots, 2n}_m
    \]
are tight $\calT(n)$-universal.
\end{theorem}

\smallskip

\section{Tight $\calT(n)$-universality of $F$ and $G$ for $n \geq 27$}\label{Sec:FromT(27)}

In this section, we prove the tight $\calT(n)$-universality corresponding to the yellow region in Figure~\ref{fig:CS(m,n)}.
In fact, we prove the result that, for $n \geq 27$, the
$m$-gonal forms
    \[
    F = \qf{n, n, n+1, \ldots, 2n-1}_m \quad \text{and} \quad G = \qf{n, n+1, \ldots, 2n}_m
    \]
are tight $\calT(n)$-universal whenever $\displaystyle 3 \leq m \leq \lfloor \tfrac{3n+1}{2} \rfloor$.
The idea is to construct sufficiently long intervals of consecutive integers represented by suitable ``head'' and ``tail'' parts of $F$ and $G$, and then to show that these intervals overlap.

To facilitate the analysis of the $m$-gonal forms $F$ and $G$, we decompose each form into a ``head'' part and a ``tail'' part as follows:
    \begin{align*}
    F_h &= nP_m(x_1)+(n+1)P_m(x_2)+(n+2)P_m(x_3)+\cdots+(n+6)P_m(x_7),\\
    F_t &=  nP_m(x_0)+(n+7)P_m(x_8)+(n+8)P_m(x_9)+\cdots+(2n-1)P_m(x_n), \\
    G_h &= nP_m(x_0)+(n+1)P_m(x_1)+(n+2)P_m(x_2)+\cdots+(n+6)P_m(x_6),\\
    G_t &= (n+7)P_m(x_7)+(n+8)P_m(x_8)+(n+9)P_m(x_9)+\cdots+2nP_m(x_n).
    \end{align*}
With these notations, the $m$-gonal forms can be rewritten as:
    \begin{align*}
    F &= \underbrace{\qf{n, n+1, n+2, \ldots, n+6}_m}_{= F_h(x_1, x_2, x_3, \ldots, x_7)}
         +  \underbrace{\qf{n, n+7, n+8, \ldots, 2n-1}_m}_{= F_t(x_0,x_8,x_9, \ldots, x_n)}, \\
    G &= \underbrace{\qf{n, n+1, n+2, \ldots, n+6}_m}_{= G_h(x_0,x_1, x_2, x_3, \ldots, x_6)}
          + \underbrace{\qf{n+7, n+8, n+9, \ldots, 2n}_m}_{= G_t(x_7,x_8, x_9,\ldots, x_n)}.
    \end{align*}
To study the representability of integers by $F$ and $G$, we introduce the following sets of integers represented by the tail parts:
    \begin{align*}
    V_i  &=\{ F_t(x_0,x_8,x_9, \ldots, x_n) \,|\,  x_j=0, 1, \text{ and } x_0+\sum_{j=8}^{n} x_j = i \},   \\
    W_i &=\{ G_t(x_7,x_8, x_9,\ldots, x_n) \,|\, x_j=0,1, \text{ and } \sum_{j=7}^{n}x_j=i \}
    \end{align*}
for each integer $2 \leq i \leq n-8$, and define the unions
    \[
    V = \bigcup_{i=2}^{n-8} V_i  \quad \text { and } \quad W= \bigcup_{i=2}^{n-8} W_i.
    \]
The following Table \ref{Tbl:integers represented by F_t, G_t} illustrates the sets $V_i$ and $W_i$ respectively, for $2 \leq i \leq n-8$.
\begin{table}[htbp]\begin{small}
\caption{Sets of integers represented by $F_t$ and $G_t$}\label{Tbl:integers represented by F_t, G_t}
\begin{tabular}{c|l|c|l}
\hline \noalign{\smallskip}
\multicolumn{2}{l|}{\hspace{15ex}$F_t$} & \multicolumn{2}{l}{\hspace{16ex}$G_t$} \\
\noalign{\smallskip} \hline \noalign{\smallskip}
\multirow{6}{*}{$V_2$} & $2n+7 = F_t(1,1,0,0, \ldots, 0)$   & \multirow{6}{*}{$W_2$} & $2n+15 = G_t(1,1,0,0, \ldots, 0)$  \\
                                      & $2n+8 = F_t(1,0,1,0, \ldots, 0)$   &                                         & $2n+16 = G_t(1,0,1,0, \ldots, 0)$  \\
                                      & $2n+9 = F_t(1,0,0,1, \ldots, 0)$   &                                         & $2n+17 = G_t(1,0,0,1, \ldots, 0)$  \\
                                      & \hspace{8ex}$\vdots$                   &                                         & \hspace{9ex}$\vdots$   \\
                                      & $4n-4 = F_t(0,0,0, \ldots, 1,0,1)$  &                                        & \hspace{1ex}$4n-2 = G_t(0,0,0, \ldots, 1,0,1)$ \\
                                      & $4n-3 = F_t(0,0,0, \ldots, 0,1,1)$  &                                        & \hspace{1ex}$4n-1 = G_t(0,0,0, \ldots, 0,1,1)$ \\
\noalign{\smallskip} \hline \noalign{\smallskip}
\multirow{6}{*}{$V_3$} & $3n+15 = F_t(1,1,1,0,0, \ldots, 0)$ & \multirow{6}{*}{$W_3$} & $3n+24 = G_t(1,1,1,0,0, \ldots, 0)$  \\
                                      & $3n+16 = F_t(1,1,0,1,0, \ldots, 0)$ &                                         & $3n+25 = G_t(1,1,0,1,0, \ldots, 0)$  \\
                                      & $3n+17 = F_t(1,1,0,0,1, \ldots, 0)$ &                                         & $3n+26 = G_t(1,1,0,0,1, \ldots, 0)$  \\
                                      & \hspace{9ex}$\vdots$                      &                                         & \hspace{9ex}$\vdots$   \\
                                      &  \hspace{1ex}$6n-7 = F_t(0,0, \ldots, 1, 0,1,1)$  &                   & \hspace{1ex}$6n-4 = G_t(0,0, \ldots, 1, 0,1,1)$ \\
                                      &  \hspace{1ex}$6n-6 = F_t(0,0, \ldots, 0, 1,1,1)$  &                   & \hspace{1ex}$6n-3 = G_t(0,0, \ldots, 0, 1,1,1)$ \\
\noalign{\smallskip} \hline \noalign{\smallskip}
$\vdots$                       & \hspace{9ex}$\vdots$                       & $\vdots$                         & \hspace{9ex}$\vdots$   \\
\noalign{\smallskip} \hline \noalign{\smallskip}
\multirow{7}{*}{$V_{n-8}$} & $ \tfrac{1}{2}(3n^2-21n-36)$      & \multirow{7}{*}{$W_{n-8}$}  & $ \tfrac{1}{2}(3n^2-19n-40)$ \\
                                            & \hspace{9ex}$= F_t(1, 1, \ldots, 1, 0,0)$ &                                  & \hspace{9ex}$= G_t(1, 1, \ldots, 1, 0,0)$  \\
                                            & \hspace{9.5ex}$\vdots$              &                                                 & \hspace{9.5ex}$\vdots$   \\
                                            & $ \tfrac{1}{2}(3n^2-17n-58)$      &                                                 & $ \tfrac{1}{2}(3n^2-15n-74)$ \\
                                            & \hspace{9ex}$= F_t(0,1,0,1, \ldots, 1)$  &                                    & \hspace{9ex}$= G_t(0,1,0,1, \ldots, 1)$  \\
                                            & $ \tfrac{1}{2}(3n^2-17n-56)$      &                                                 & $ \tfrac{1}{2}(3n^2-15n-72)$ \\
                                            & \hspace{9ex}$= F_t(0,0,1,1, \ldots, 1)$  &                                    & \hspace{9ex}$= G_t(0,0,1,1, \ldots, 1)$  \\
\noalign{\smallskip} \hline \noalign{\smallskip}
\end{tabular}
\end{small}
\end{table}

\smallskip

\begin{lemma}\label{Lem:03}
Let $n \geq 27$. \\
\noindent \rm{(1)} The sets
        \begin{align*}
        V &= \bigcup_{i=2}^{n-8} V_i =\{ 2n+7, \, 2n+8, \, 2n+9, \, \ldots, \,  \tfrac{1}{2}(3n^2-17n-56) \},\\
        W &= \bigcup_{i=2}^{n-8} W_i = \{ 2n+15, \, 2n+16, \, 2n+17, \, \ldots, \,  \tfrac{1}{2}(3n^2-15n-72) \}
        \end{align*}
        are intervals of consecutive integers represented by the tail part $F_t$  and $G_t$ respectively.\\

\noindent \rm{(2)}  If an integer $\alpha$ is represented by the head part $F_h$ and $G_h$ respectively, then all the shifted intervals of consecutive integers
        \begin{align*}
        \alpha + V
        &= \{ \,\alpha + \beta \mid \beta \in V\,\}
          \,= \{ \alpha + 2n+7, \, \alpha + 2n+8, \,  \ldots, \,  \alpha + \tfrac{1}{2}(3n^2-17n-56) \}, \\
        \alpha + W
        &= \{ \,\alpha + \beta \mid \beta \in W\,\}
          = \{ \alpha + 2n+15, \, \alpha + 2n+16, \,  \ldots, \,  \alpha + \tfrac{1}{2}(3n^2-15n-72) \}
        \end{align*}
    are represented by the $m$-gonal forms $F$ and $G$ respectively.
\end{lemma}

\begin{proof}

\noindent \rm{(1)}
    For all $2 \leq i \leq n-8$, it is readily checked that $V_i$ and $W_i$ consist of consecutive integers represented by the tail part $F_t$ and $G_t$  respectively, since $P_m(1) = 1$.
    The only possible gaps could occur at the transition from $V_i$ to $V_{i+1}$ and  $W_i$ to $W_{i+1}$.
    For each $2 \leq i \leq n-9$, the last terms of $V_i$ and $W_i$ are
        \begin{align*}
        F_t(0, \ldots, 0, \underbrace{1, \ldots, 1}_{i \text{ terms }})
        &=\sum_{k=0}^{i-1} (2n-i+k) = \tfrac{1}{2}(4ni-i^2-i),\\
        G_t(0, \ldots, 0, \underbrace{1, \ldots, 1}_{i \text{ terms }})
        &=\sum_{k=1}^{i} (2n-i+k) = \tfrac{1}{2}(4ni-i^2+i)
        \end{align*}
    and the initial terms of $V_{i+1}$ and $W_{i+1}$ are
        \begin{align*}
        F_t(\hspace{-1ex}\underbrace{1, \ldots, 1}_{ (i+1) \text{ terms }}\hspace{-1ex}, 0, \ldots, 0)
        &= n+  \sum_{k=1}^{i} (n+6+k) = \tfrac{1}{2}(2n(i+1) + i^2 + 13i),\\
        G_t(\hspace{-1ex}\underbrace{1, \ldots, 1}_{ (i+1) \text{ terms }}\hspace{-1ex}, 0, \ldots, 0)
        &= \sum_{k=1}^{i+1} (n+6+k) = \tfrac{1}{2}(2n(i+1) +i^2 +15i +14)
        \end{align*}
    respectively.
    Since $n \geq 27$, by comparing the last element of $V_i$ with the first element of $V_{i+1}$ explicitly, the difference between the last integer in $V_i$ and the first integer in $V_{i+1}$ is
        \begin{align*}
           & F_t(0, \ldots, 0, \underbrace{1, \ldots, 1}_{i \text{ terms }}) -  F_t(\hspace{-1ex}\underbrace{1, \ldots, 1}_{ (i+1) \text{ terms }}\hspace{-1ex}, 0, \ldots, 0) \\
        = & -i^2 +ni -7i -n
        = -(i-(n-9))(i-2) + n-18 > 0
        \end{align*}
    for all $2 \leq i \leq n-9$. Hence no gaps occur in the union.
    With the same calculation, the difference between the last integer of $W_i$ and the first integer of $W_{i+1}$ is positive.
    Hence  the unions $\displaystyle V = \bigcup_{i=2}^{n-8} V_i$  and $\displaystyle W= \bigcup_{i=2}^{n-8} W_i$  are intervals of consecutive integers represented by $F_t$ and $G_t$, respectively.\\

\noindent \rm{(2)}
    Since $F = F_h + F_t$ and $G=G_h+G_t$,  adding any $\alpha$ from the head part to all elements of $V$ (resp. W) produces consecutive intervals $\alpha + V$ (resp. $\alpha + W$) represented by $F$ (resp. $G$).
\end{proof}

\smallskip

For convenience in presenting Lemma~\ref{Lem:04}, we note that
    \[
    P_m(1)=1, \quad P_m(-1)=m-3, \quad P_m(2)=m, \quad P_m(-2)=3m-8
    \]
and let
    \[
    A_q = F_h(v_q), \quad B_q = F_h(w_q)
    \]
where the vectors $v_q$ and $w_q$ are given in Table~\ref{Tbl:Some integers represented by F_h and G_h} for all $q = 1, 2, \ldots, 10$.
Since $F_h$ and $G_h$ have the same coefficients, we also have
    \[
    A_q=G_h(v_q), \quad  B_q=G_h(w_q)
    \]
for all $q = 1, 2, \ldots, 10$.
\begin{table}[htbp]
\begin{small}
\caption{Some integers represented by $F_h$ and $G_h$}\label{Tbl:Some integers represented by F_h and G_h}
\begin{tabular}{p{1ex}|p{17ex}|p{28.5ex}|p{18.5ex}|p{15ex}}
\hline \noalign{\smallskip}
$q$   & \hspace{2ex}$A_q = F_h(v_q)$  & \hspace{12ex}$v_q$  & \hspace{2ex}$B_q = F_h(w_q)$  &  \hspace{6ex}$w_q$\\
 \noalign{\smallskip}\hline \noalign{\smallskip}
$1$   & $n(m-3)$                &$(-1,0,0,0,0,0,0)$        &$(n+6)m+6n+15$   &$(1,1,1,1,1,1,2)$\\
$2$   & $(2n+1)(m-3)$       &$(-1,-1,0,0,0,0,0)$       &$(2n+11)m+2n+7$ &$(0,0,0,1,1,2,2)$\\
$3$   & $3(n+3)(m-3)$       &$(0,0,-1,-1,-1,0,0)$      &$(3n+15)m$           &$(0,0,0,0,2,2,2)$\\
$4$   & $4(n+3)(m-3)$       &$(0,-1,-1,0,-1,-1,0)$     &$(n+3)(4m+3)$      &$(1,2,2,1,2,2,1)$\\
$5$   & $5(n+3)(m-3)$       &$(0,-1,-1,-1,-1,-1,0)$    &$(n+3)(5m+2)$      &$(1,2,2,2,2,2,1)$\\
$6$   & $6(n+3)(m-3)$       &$(-1,-1,-1,0,-1,-1,-1)$  &$(n+3)(6m+1)$      &$(2,2,2,1,2,2,2)$\\
$7$   & $7(n+3)(m-3)$       &$(-1,-1,-1,-1,-1,-1,-1)$ &$(n+3)(7m)$          &$(2,2,2,2,2,2,2)$\\
$8$   & $(n+3)(8m-22)$     &$(0,-2,-1,0,-1,-2,0)$     &$(n+3)(8m-3)$       &$(1,2,3,1,3,2,1)$\\
$9$   & $(n+3)(9m-26)$     &$(-1,-1,-1,-2,-1,-1,-1)$ &$(n+3)(9m-6)$      &$(0,3,2,2,2,3,0)$\\
$\hspace{-0.5ex}10$ & $(n+3)(10m-28)$  &$(-1,-1,-2,0,-2,-1,-1)$   &$(n+3)(10m-6)$    &$(3,2,2,0,2,2,3)$\\
 \noalign{\smallskip}\hline \noalign{\smallskip}
\end{tabular}
\end{small}
\end{table}

\smallskip

\begin{lemma}\label{Lem:04}
Let $n \geq 27$ and $\displaystyle 3 \leq m \leq \lfloor \tfrac{3n+1}{2} \rfloor$.
Then the $m$-gonal forms $F$ and $G$ represent all intervals of consecutive integers $\displaystyle \bigcup_{q=1}^{10} I_q $ and $\displaystyle \bigcup_{q=1}^{10} J_q$, respectively, where
    \[
    I_q = (A_q + V) \cup (B_q + V) \quad \text{ and } \quad J_q = (A_q + W) \cup (B_q + W)
    \]
for $q= 1,2, \ldots, 10$.
\end{lemma}

\begin{proof}
We prove that $\displaystyle \bigcup_{q=1}^{10} I_q$  is a union of intervals of consecutive integers represented by the $m$-gonal form $F$  through the following steps:

\Step{1)} By Lemma \ref{Lem:03}, the union
    \[
    V = \bigcup_{i=2}^{n-8} V_i
       = \left\{ ~2n+7, \quad 2n+8, \quad 2n+9, \quad \ldots, \quad  \tfrac{1}{2}(3n^2 - 17n - 56)~ \right\}
    \]
is an interval of consecutive integers represented by $F_t$.

\Step{2)} For each $q = 1, 2, \ldots, 10$, both $A_q = F_h(v_q)$ and $B_q = F_h(w_q)$ are integers represented by the head part $F_h$.
By Lemma \ref{Lem:03},  the shifted sets
    \begin{align*}
    A_q + V  &= \{ ~A_q + \beta ~\mid~ \beta \in V~ \}\\
                    &= \left\{ ~A_q+2n+7, \quad A_q+2n+8, \quad \ldots, \quad   A_q+\tfrac{1}{2}(3n^2 - 17n - 56)~ \right\},\\
    B_q + V  &= \{ ~B_q + \beta ~\mid~ \beta \in V~ \}\\
                    &= \left\{ ~B_q+2n+7, \quad B_q+2n+8, \quad \ldots, \quad   B_q+\tfrac{1}{2}(3n^2 - 17n - 56)~ \right\}
    \end{align*}
are represented by $F$ and each forms an interval of consecutive integers.

\Step{3)} We show that the last integer in $A_q + V$ and the first integer in $B_q + V$ overlap or are consecutive, i.e.,
    \[
    \bigl[ A_q + \tfrac{1}{2}(3n^2 - 17n - 56) \bigr] - \bigl[ B_q + 2n + 7 \bigr] \geq 0.
    \]
For $q=1$, we have
    \begin{align*}
    &\bigl[ A _1+ \tfrac{1}{2}(3n^2-17n-56) \bigr] - \bigl[ B_1+2n+7 \bigr]\\
    &= \bigl[ n(m-3)+\tfrac{1}{2}(3n^2-17n-56) \bigr] - \bigl[  (n+6)m+6n+15+(2n+7) \bigr] \\
    &= \tfrac{1}{2}(3n^2 -39n -100) -6m \geq \tfrac{1}{2}(3n^2 -39n -100) - 6 \cdot \tfrac{1}{2}(3n+1)\\
    &= \tfrac{1}{2}(3n^2-57n -106) \geq 0 \quad \text{ as } \quad n \geq 27.
    \end{align*}
Similar computations hold for $q = 2, 3, \ldots, 10$.
Hence, for each $q=1, 2, \ldots, 10$, the set
    \[
    I_q = (A_q + V) \cup (B_q + V)
    \]
forms an interval of consecutive integers.

\Step{4)} We verify that for $q = 1, 2, \ldots, 9$, the last integer in $B_q + V$ and the first integer in $A_{q+1} + V$ overlap or are consecutive, i.e.,
    \[
    \bigl[ B_q + \tfrac{1}{2}(3n^2 - 17n - 56) \bigr] - \bigl[ A_{q+1} + 2n + 7 \bigr] \geq 0.
    \]
For $q=1$, this difference is
    \begin{align*}
    &\bigl[ B_1+ \tfrac{1}{2}(3n^2-17n-56) \bigr] - \bigl[ A_2+2n+7 \bigr]\\
    &= \bigl[ (n+6)m +6n+15+\tfrac{1}{2}(3n^2-17n-56) \bigr] - \bigl[ (2n+1)(m-3)+(2n+7) \bigr]\\
    &= \tfrac{1}{2}(3n^2 +3n -34) -(n-5)m  \geq \frac{1}{2}(3n^2 +3n -34) -(n-5)\cdot \tfrac{1}{2}(3n+1)\\
    &= \tfrac{1}{2}(17n-29) \geq 0 \quad \text{ as } \quad n \geq 27.
    \end{align*}
Similar arguments apply for all $q=2, 3, \ldots, 9$.

\Step{5)} Combining these results, we conclude that the union
    \begin{align*}
    \bigcup_{q=1}^{10} I_q
    &= \bigcup_{q=1}^{10} (A_q + V \cup B_q + V)\\
    &= \left\{
         \begin{array}{l}
         n(m-3)+2n+7, \quad  n(m-3)+2n+8,  \\
         \hspace{15ex} \ldots, \quad  (n+3)(10m-6)+\frac{1}{2}(3n^2-17n-56)
         \end{array}
         \right\}
    \end{align*}
is an interval of consecutive integers starting from $n(m-3) + 2n + 7$ up to $(n+3)(10m - 6) + \tfrac{1}{2}(3n^2 - 17n - 56)$ represented by the $m$-gonal form $F$.

By similar reasoning, the union
    \begin{align*}
    \bigcup_{q=1}^{10} J_q
    &= \bigcup_{q=1}^{10} (A_q + W \cup B_q + W)\\
    &= \left\{
         \begin{array}{l}
         n(m-3)+(2n+15), \quad  n(m-3)+(2n+16),  \\
         \hspace{15ex} \ldots, \quad  (n+3)(10m-6)+\frac{1}{2}(3n^2-15n-72)
         \end{array}
         \right\}
    \end{align*}
is also an interval of consecutive integers represented by the $m$-gonal form $G$.
\end{proof}

\smallskip

\begin{lemma}\label{Lem:05}
Let $q\geq11$.
Then there exist $r_q\in\{ 5, 9 \}$, $s_q\in\{ 12, 16 \}$ such that the $m$-gonal form $f(\mathbf{x})=\qf{1,2,2,2}_m$ represents integers
    \[
    a_q=(m-2)q-r_q, \qquad b_q=(m-2)q+s_q.
    \]
Consequently, if we let
    \[
    A_q=(n+3)a_q, \qquad B_q=(n+3)b_q,
    \]
then both integers $A_q$ and $B_q$ are represented by each of the head parts $F_h$ and $G_h$.
\end{lemma}

\begin{proof}
We first prove that $a_q$ and $b_q$ are represented by $f=\qf{1,2,2,2}_m$.
The argument in Section~\ref{Sec:qf{1222}_m} remains valid for arbitrary integers $q$ and $r$.
Thus it suffices to find a positive definite binary lattice $A[0]$ satisfying
    \[
    | A[0] | = \det \begin{pmatrix} ~7 & ~r~ \\ ~r & ~2q+r~ \end{pmatrix}
               = 14q + 7r - r^2
               \neq 4^s(16t+14)
    \]
for all nonnegative integers $s$ and $t$.
Here, the restriction $0 \leq r \leq m-3$ used in Section~\ref{Sec:qf{1222}_m} is needed only to choose a canonical representative modulo $m-2$.

Note that
    \begin{align*}
    | A[0] | = \begin{cases} ~D_1:= 14q-60   & \text{ if } r = -5, \, 12, \\
                                        ~D_2:= 14q-144 & \text{ if } r = -9, \, 16.
                  \end{cases}
    \end{align*}
Since $q \geq 11$, both $D_1$ and $D_2$ are positive.

If $D_1$ is not of the form $4^s(16t+14)$, we take $r_q=5$, $s_q=12$.
Then the corresponding determinants $| A[0] |$ for $a_q=(m-2)q-r_q$ and $b_q=(m-2)q+s_q$ are both equal to $D_1$.
Suppose now that $D_1=4^s(16t+14)$ for some nonnegative integers $s$ and $t$.
Since $D_2 = D_1 - 84$, we have
    \[
    D_2 = 14q - 144 = 4^s(16t+14) - 84
              \equiv \begin{cases} ~10 \pmod{16} & \text{ if } s= 0,\\
                                              ~~4 \pmod{8}   & \text{ if } s\geq 1,
                         \end{cases}
    \]
Thus $D_2$ is not of the form $4^u(16v+14)$ for any nonnegative integers $u$ and $v$.
In this case, we take $r_q=9$, $s_q=16$.
Then the corresponding determinants $| A[0] |$ for $a_q$ and $b_q$ are both equal to $D_2$.
In either case, the relevant binary lattices are positive definite.
Therefore, by the representation criterion in Section~\ref{Sec:qf{1222}_m}, both $a_q$ and $b_q$ are represented by the $m$-gonal form $f=\qf{1,2,2,2}_m$.

If an integer $\alpha$ is represented by  $f(\mathbf{x}) = \qf{1,2,2,2}_m$, then $(n+3)\alpha$ is represented by both $F_h$ and $G_h$, since
    \begin{align*}
    (n+3) \qf{1,2,2,2}_m
    &= (n+3) \,[ P_m(y_1) + 2P_m(y_2) + 2P_m(y_3) + 2P_m(y_4)]\\
    &= nP_m(y_2) + (n+1)P_m(y_3) + (n+2)P_m(y_4) + (n+3)P_m(y_1) \\
    &\hspace{12.5ex} + (n+4)P_m(y_4) + (n+5)P_m(y_3) + (n+6)P_m(y_2) \\
    &= F_h (y_2, y_3, y_4, y_1, y_4, y_3, y_2) = G_h (y_2, y_3, y_4, y_1, y_4, y_3, y_2).
    \end{align*}
Applying this to $\alpha=a_q$ and $\alpha=b_q$ shows that
    \[
    A_q=(n+3)a_q\quad\text{and}\quad B_q=(n+3)b_q
    \]
are represented by $F_h$ and $G_h$ as claimed.\\
\end{proof}

\smallskip

\begin{lemma}\label{Lem:06}
Let $n \geq 27$ and $\displaystyle 3 \leq m \leq \lfloor \tfrac{3n+1}{2} \rfloor$.
Then the $m$-gonal forms $F$ and $G$ represent all intervals of consecutive integers $\displaystyle \bigcup_{q \geq 11} I_q $ and $\displaystyle \bigcup_{q \geq 11} J_q$, respectively, where
    \begin{align*}
    I_q = (A_q + V) \cup (B_q + V)  \quad \text{ and } \quad
    J_q = (A_q + W) \cup (B_q + W)
    \end{align*}
for $q \geq 11$.
\end{lemma}

\begin{proof}
The argument parallels Lemma~\ref{Lem:04}.\\

\Step{1)} For $q \geq 11$ and the integers $a_q$, $b_q$ as defined in Lemma \ref{Lem:05}, we know that
    \begin{align*}
    a_q - b_q &= \bigl[ (m-2)q - r_q \bigr] - \bigl[  (m-2)q + s_q \bigr] \\
                    &= - r_q - s_q \geq -9 -16 =-25,\\
    b_q - a_{q+1} &= \bigl[ (m-2)q + s_q \bigr] - \bigl[ (m-2)(q+1) - r_{q+1} \bigr]\\
                          &= r_{q+1} + s_q -(m-2) \geq 5+12-(m-2) = -m+19.
    \end{align*}
Moreover, both integers $A_q$, $B_q$ are represented by $F_h$ and the shifted sets
    \begin{align*}
    A_q + V  &= \{ ~A_q + \beta ~\mid~ \beta \in V~ \}
                   = \left\{ ~A_q+2n+7, \,  \ldots, \,  A_q+\tfrac{1}{2}(3n^2 - 17n - 56)~ \right\},\\
    B_q + V  &= \{ ~B_q + \beta ~\mid~ \beta \in V~ \}
                   = \left\{ ~B_q+2n+7, \,  \ldots, \,   B_q+\tfrac{1}{2}(3n^2 - 17n - 56)~ \right\}
    \end{align*}
are intervals of consecutive integers represented by $F$ by Lemma \ref{Lem:03}.

\Step{2)} For each $q \geq 11$, we can show that the last integer in $A_q + V$ and the first integer in $B_q + V$ overlap or are consecutive, i.e.,
    \begin{align*}
    & \bigl[ A _q+ \tfrac{1}{2}(3n^2-17n-56) \bigr]- \bigl[  B_q+2n+7 \bigr]\\
    &= \tfrac{1}{2} (3n^2 - 21n - 70) + (n+3)(a_q - b_q)  \geq \tfrac{1}{2} (3n^2 - 21n - 70) - 25(n+3) \\
    &= \tfrac{1}{2}(3 n^2 - 71 n - 220) \geq 0 \quad \text{ as } \quad n \geq 27.
    \end{align*}
Therefore their union,
    \[
    I_q = (A_q + V) \cup (B_q + V),
    \]
is also a consecutive interval represented by $F$ for all $q \geq 11$.

\Step{3)} We verify that for $q \geq 11$, the last integer in $B_q + V$ and the first integer in $A_{q+1} + V$ overlap or are consecutive,  i.e.
    \begin{align*}
    &\bigl[  (n+3)b _q+ \tfrac{1}{2}(3n^2-17n-56) \bigr] - \bigl[ (n+3)a_{q+1} + 2n+7 \bigr]\\
    &= \tfrac{1}{2}(3 n^2 - 21 n - 70) + (n+3)(b_q - a_{q+1})\\
    &\geq \tfrac{1}{2} (3n^2 - 21n - 70) - (n+3)(m-19) \\
    &\geq \tfrac{1}{2} (3n^2 - 21n - 70) - (n+3)(\tfrac{1}{2}(3n+1)-19)
    = \tfrac{1}{2}(7 n + 41)  \geq 0
    \end{align*}
as $n \geq 27$.
Therefore, the union
    \begin{align*}
    \bigcup_{q \geq 11} I_q
    &= \bigcup_{q \geq 11} (A_q + V \cup B_q + V)
    \end{align*}
is an interval of consecutive integers represented by the $m$-gonal form $F$.
An analogous calculation shows the same for the union $\displaystyle \bigcup_{q \ge 11} J_q$, which is an interval of consecutive integers represented by the $m$-gonal form $G$.
\end{proof}

\smallskip

\begin{lemma}\label{Lem:07}
Let $n \geq 27$ and $\displaystyle 3 \leq m \leq \lfloor \tfrac{3n+1}{2} \rfloor$.
The $m$-gonal forms
    \[
    F=\qf{n, n, n+1, \ldots, 2n-1}_m \text{ and } G=\qf{n, n+1, n+2,  \ldots, 2n}_m
    \]
represent the intervals of consecutive integers
   \begin{align*}
   I_0 = &\left\{ ~n, \, n+1, \, n+2,  \, \ldots, \,  \tfrac{1}{2}(3n^2-n-2), \,  \tfrac{1}{2}(3n^2-n)~ \right\}\\
            &\bigcup \left\{ ~n(m-3)+2n+1, \, n(m-3)+2n+2, \, \ldots, n(m-3)+2n+6~ \right\}
   \end{align*}
and
   \[
   \hspace{-15ex} J_0 = \left\{ ~n, \, n+1, \, n+2,  \, \ldots, \,  \tfrac{1}{2}(3n^2+n-2), \,  \tfrac{1}{2}(3n^2+n)~ \right\}
   \]
respectively.
\end{lemma}

\begin{proof}
We prove this Lemma through the following steps using a symmetry relation.

\smallskip
\Step{1)}
Since $P_m(1) =1$, we have the relation
    \begin{align*}
    & F(x_0, x_1, x_2, \ldots, x_n) + F(1-x_0, 1-x_1, 1-x_2, \ldots, 1-x_n) \\
    &= F(1, 1, 1, \ldots, 1) \\
    &= n + n + (n+1) + \ldots + (2n-1) = \tfrac{1}{2}(3n^2+n),
    \end{align*}
with $x_j =0, 1$ and $0 \leq j  \leq n$.
If $F$ represents an integer $k$ with
    \[
     F(x_0, x_1, x_2, \ldots, x_n) = k, \quad x_j=0, 1, \quad 0 \leq j  \leq n,
     \]
then $F$ represents an integer
    \[
    F(1-x_0, 1-x_1, 1-x_2, \ldots, 1-x_n) = \tfrac{1}{2}(3n^2+n) -k.
    \]

\smallskip
\Step{2)}
On the other hand, $F(x_0, x_1, x_2, \ldots, x_n)$ represents all consecutive integers from
$n$ through $2n+6$ with $x_j =0, 1$ and $0 \leq j  \leq n$.
By Lemma \ref{Lem:03},  $F(x_0, x_1, x_2, \ldots, x_n)$ represents intervals of consecutive integers
    \[
    V = \bigcup_{i=2}^{n-8} V_i =\{ 2n+7, \, 2n+8, \,  \ldots, \,  \tfrac{1}{2}(3n^2-17n-56) \}
    \]
with $x_j =0, 1$ and $0 \leq j  \leq n$.
Therefore  $F$ represents all consecutive integers
    \[
    n, \, n+1,\, n+2, \, \ldots, \, 2n+6, \, 2n+7, \, 2n+8, \, \ldots,  \,  \tfrac{1}{2}(3n^2-17n-56).
    \]
The above symmetry relation forces $F$ to represent all consecutive integers
    \begin{align*}
    \tfrac{1}{2}(3n^2+n) - \tfrac{1}{2}(3n^2-17n-56) &= 9n+28, \\
                                                                                    & \,\,\, \vdots \\
    \tfrac{1}{2}(3n^2+n) - (n+1) &= \tfrac{1}{2}(3n^2 -n-2), \\
    \tfrac{1}{2}(3n^2+n) - n &= \tfrac{1}{2}(3n^2 -n).
    \end{align*}
Since
    \[
    \tfrac{1}{2}(3n^2-17n-56)  -  (9n+28) = \tfrac{1}{2}(3n^2-35n-112) >0
    \]
for all $n \geq 27$, the $m$-gonal form $F$ represents an interval of consecutive integers
   \[
   \left\{ ~n, \, n+1, \, n+2,  \, \ldots, \,  \tfrac{1}{2}(3n^2-n-2), \,  \tfrac{1}{2}(3n^2-n)~ \right\}.
   \]

\smallskip
\Step{3)}
By a simple observation, the $m$-gonal form $F$ represents all  consecutive integers from
    \[
    n(m-3)+ 2n+1 \quad \text{ through } \quad n(m-3)+ 2n+6
    \]
as $nP_m(-1)+nP_m(1) = n(m-3)+n$ and $(n+j)P_m(1)=n+j$.
By the following careful calculation
    \begin{align*}
    \tfrac{1}{2}(3n^2-n) - [~ n(m-3)+ 2n+1~]
    &= \tfrac{1}{2}(3n^2 + n -2) -mn   \\
    &\geq \tfrac{1}{2}(3n^2 + n -2) - \tfrac{1}{2}(3n+1) \cdot n = -1,
    \end{align*}
the $m$-gonal form $F$ represents an interval of consecutive integers
   \begin{align*}
   I_0 = &\left\{ ~n, \, n+1, \, n+2,  \, \ldots, \,  \tfrac{1}{2}(3n^2-n-2), \,  \tfrac{1}{2}(3n^2-n)~ \right\}\\
            &\bigcup \left\{ ~n(m-3)+2n+1, \, n(m-3)+2n+2, \, \ldots, n(m-3)+2n+6~ \right\}.
   \end{align*}
Similarly, we obtain that the $m$-gonal form $G$ represents an interval of consecutive integers
   \[
   J_0 = \left\{ ~n, \, n+1, \, n+2,  \, \ldots, \,  \tfrac{1}{2}(3n^2+n-2), \,  \tfrac{1}{2}(3n^2+n)~ \right\}.
   \]
\end{proof}

\smallskip

\begin{theorem}\label{Thm:05}
Let $n \geq27$.
If $3 \leq m \leq \lfloor \frac{3n+1}{2} \rfloor$, then the $m$-gonal forms
    \[
    F=\qf{n, n, n+1, \ldots, 2n-1}_m \,\text{ and }\, G=\qf{n, n+1, n+2,  \ldots, 2n}_m
    \]
are tight $\calT(n)$-universal.
\end{theorem}

\begin{proof}
Since every coefficient of $F$ and $G$ is at least $n$ and $P_m(x)\geq0$ for every $x \in \Z$, neither form represents a positive integer less than $n$.
Furthermore, by the preceding Lemma \ref{Lem:07}, Lemma \ref{Lem:04}, Lemma \ref{Lem:06},  $F$ represents all intervals of consecutive integers
    \begin{align*}
    I_0
    &=\left\{ ~n, \, n+1, \, n+2,  \, \ldots, \,  \tfrac{1}{2}(3n^2-n-2), \,  \tfrac{1}{2}(3n^2-n)~ \right\}\\
    &\hspace{2ex} \bigcup \left\{ ~n(m-3)+2n+1, \,  n(m-3)+2n+2, \, \ldots, n(m-3)+2n+6~ \right\}\\
    \bigcup_{q=1}^{10} I_q
    &= \left\{
         \begin{array}{ll}
         n(m-3)+2n+7, \,  n(m-3)+2n+8,  \\
         \hspace{15ex} \ldots, \,(n+3)(10m-6)+\frac{1}{2}(3n^2-17n-56)
         \end{array}
         \right\},\\
    \bigcup_{q \geq 11} I_q
    &= \left\{ ~A_{11}+2n+7, \quad A_{11}+2n+8,  \,   \ldots  \hspace{9ex} \right\}.
    \end{align*}
Since
    \begin{align*}
    & \bigl[ (n+3)(10m-6)+\tfrac{1}{2}(3n^2-17n-56) \bigr] - \bigl[ A_{11}+(2n+7) \bigr]\\
    &\geq {(n+3)(10m-6)+\tfrac{1}{2}(3n^2-17n-56)} - ( (n+3)(11(m-2)-5)+(2n+7) )\\
    &= \tfrac{1}{2}(3n^2+21n+56) -(n+3)m \\
    &\geq  \tfrac{1}{2}(3n^2+21n+56)-  (n+3)  \cdot \tfrac{3n+1}{2} = \tfrac{1}{2}(11n+53) \geq 0,
    \end{align*}
the last integer in $\displaystyle \bigcup_{q=1}^{10} I_q$ and the first integer in $\displaystyle \bigcup_{q \geq 11} I_q$ overlap.
Therefore, $\displaystyle \calT(n) = \bigcup_{q \geq 0} I_q$ and we have proven that $F$ is tight $\calT(n)$-universal.
By a similar argument, we have proven that $G$ is tight $\calT(n)$-universal.
\end{proof}

\smallskip

\section{Optimality of the Universality Range}\label{Sec:Optimality}

In this section, we establish the optimality of the universality range corresponding to the red region in Figure~\ref{fig:CS(m,n)}.
More precisely, we show that for $n \geq 8$ and $m \geq \frac{3n+2}{2}$, at least one of the $m$-gonal forms
    \[
    F=\qf{n, n, n+1, \ldots, 2n-1}_m \quad \text{ and } \quad G=\qf{n,n+1,n+2,\ldots, 2n}_m
    \]
fails to be tight $\calT(n)$-universal.

Throughout this section, we use the facts
    \[
    P_m(0)=0,\qquad P_m(1)=1,\qquad P_m(-1)=m-3.
    \]
Since $n \geq 8$ and $m \geq \frac{3n+2}{2}$, we have $m \geq 13$.
Hence $P_m(x)\geq m-3$ for $x \notin \{ 0,1 \}$, and $P_m(x)\geq m$ for $x \notin \{ 0, 1, -1\}$.
For notational convenience, we write
    \[
    F(\mathbf{x}) = nP_m(x_0)+\sum_{i=1}^{n}(n+i-1)P_m(x_i) \quad \text{ and } \quad
    G(\mathbf{x}) = \sum_{i=0}^n (n+i) P_m(x_i).
    \]

\smallskip

\begin{proposition}\label{Prop:non-representation-large-m}
Let $n\geq8$ and $m > \left\lfloor \frac{1}{2}(3n+3) \right\rfloor$.
Then the $m$-gonal form
    \[
    F = \qf{n,n,n+1,n+2,\ldots,2n-1}_m
    \]
does not represent the integer
    \[
    N:=\frac{1}{2}(3n^2-n) + 1.
    \]
\end{proposition}


\begin{proof}
Suppose that $F(\mathbf{x})=N$ for some $\mathbf{x} = (x_0, x_1, \ldots, x_n) \in \Z^{n+1}$.

\smallskip

First suppose that $x_i \in \{0,1\}$ for all $i$.
Since
    \[
    F(1,\ldots,1) = n + n + (n+1) + \cdots + (2n-1) = \frac{n(3n+1)}{2} = N + (n-1),
    \]
every proper subsum of the coefficients of $F$ is at most
    \[
    F(1,\ldots,1) - n < N,
    \]
whereas the full sum of the coefficients of $F$ is greater than $N$.
Thus some $x_i \notin \{0,1\}$.

If $x_i \notin \{0,1\}$ for some $i \geq 2$, then, since $m > \lfloor\frac{3n+3}{2} \rfloor$ implies $m \geq \frac{3n+4}{2}$, we have
    \[
    F(\mathbf{x})\geq(n+1)(m-3)\geq\frac{3n^2+n-2}{2}>N,
    \]
a contradiction.
Hence $x_i \in \{0,1\}$ for every $i \geq 2$, and at least one of $x_0, x_1$ lies outside $\{0,1\}$.

If $x_0 \neq -1$ and $x_1 \neq -1$, then one of them lies outside
$\{0,1,-1\}$, and hence
    \[
    F(\mathbf{x}) \geq nm >n \,\frac{3n+3}{2} > N,
    \]
a contradiction. Thus, by symmetry, we may assume that $x_0=-1$.
The remaining variables must then represent
    \[
    R:= N - nP_m(-1) = N-n(m-3)  =  \frac{n(3n+5-2m)+2}{2}.
    \]
  Since $m \geq \frac{3n+4}{2}$,
    \[
    R = \frac{n(3n+5-2m)+2}{2} \leq \frac{n(3n+5-(3n+4))+2}{2} = \frac{n+2}{2} <n.
    \]
Moreover, $R \neq 0$, since $2R \equiv 2 \pmod n$ whereas $n \nmid 2$.
However, the remaining form \text{$\qf{n,n+1,\ldots,2n-1}_m$} represents only $0$ or integers at least $n$, a contradiction.
Therefore $F$ does not represent $N$.
\end{proof}

\smallskip

\begin{proposition}\label{Prop:non-representation-even-boundary}
Let $n\geq8$ be even and let $m=\frac{1}{2}(3n+2)$.
Then the $m$-gonal form
    \[
    F=\qf{n,n,n+1,n+2,\ldots,2n-1}_m
    \]
does not represent the integer
    \[
    T := n(m-3)+2n = \frac{3}{2}n^2.
    \]
\end{proposition}

\begin{proof}
Suppose that $F(\mathbf{x})=T$ for some $\mathbf{x} = (x_0, x_1, \ldots, x_n) \in \Z^{n+1}$.

\smallskip

Suppose first that $x_i \in \{0,1\}$ for all $i$.
If at least one $x_i$ is zero, then
    \[
    F(\mathbf{x}) \leq F(1,\ldots,1)- nP_m(1) = \frac{n(3n+1)}{2} - n = \frac{n(3n-1)}{2}  < \frac{3n^2}{2} =T.
    \]
On the other hand,
    \[
    F(1,\ldots,1) = \frac{n(3n+1)}{2} > T.
    \]
Thus some $x_i \notin \{ 0,1 \}$.

If $x_i \notin \{ 0,1,-1 \}$ for some $i$, then $P_m(x_i) \geq m$ and hence
    \[
    F(\mathbf{x}) \geq nm= \frac{3n^2+2n}{2} >  \frac{3n^2}{2} =T,
    \]
which is impossible.
Hence every variable outside $\{0,1\}$ must be equal to $-1$.

Suppose that $x_i = -1$ for some $3 \leq i \leq n$.
Since $m-3 = \frac{3n-4}{2}$, we have
    \[
    F(\mathbf{x}) \geq (n+2)(m-3) = (n+2) \,\frac{3n-4}{2} = \frac{3n^2}{2} + (n-4) > \frac{3n^2}{2} =T,
    \]
because $n \geq 8$.
Thus this is impossible.
Suppose next that $x_2=-1$. Then
    \[
    0 < T-(n+1)P_m(-1) = \frac{3n^2}{2} -(n+1)\,\frac{3n-4}{2} = \frac{n+4}{2} < n.
    \]
Since all coefficients of the remaining form $\qf{n,n,n+2,\ldots,2n-1}_m$ are at least $n$, this is impossible.
Hence $x_2\neq-1$.
It remains to consider $x_0=-1$ or $x_1=-1$.
By symmetry, assume that $x_0=-1$.
Then the remaining variables must represent
    \[
    T - nP_m(-1) = [ ~n(m-3)+2n~ ] - n(m-3) = 2n.
    \]
Any variable outside $\{0,1\}$ would contribute at least $n(m-3) > 2n$.
Thus $2n$ would have to be a subsum of $n, n+1, \ldots, 2n-1$.
However, no coefficient equals $2n$, while the sum of the two smallest coefficients is $2n+1$, a contradiction.

This contradiction proves that $F$ does not represent the integer $T$.
\end{proof}

\smallskip

\begin{proposition}\label{Prop:non-representation-odd-boundary}
Let $n \geq 8$ be odd and let $m=\frac{1}{2}(3n+3)$.
Then the $m$-gonal form
    \[
    G=\qf{n, n+1, \ldots, 2n}_m
    \]
does not represent either of the consecutive integers
    \[
    T_j:=\frac{1}{2}(3n^2+n)+j, \qquad j=1, 2.
    \]
\end{proposition}

\begin{proof}
Fix $j \in \{ 1, 2 \}$ and suppose that $G(\mathbf{x})=T_j$ for some $\mathbf{x} = (x_0, x_1, \ldots, x_n) \in \Z^{n+1}$.

\smallskip

If $x_i \in\{ 0, 1 \}$ for all $i$, then
    \[
    G(1,\ldots,1) - n P_m(1) = \frac{3n^2+3n}{2} -n = \frac{3n^2+n}{2}  < T_j  < G(1, \ldots,1) =\frac{3n^2+3n}{2}.
    \]
Hence no subsum of the coefficients $\{ n, n+1, \ldots, 2n \}$ of $G$ equals $T_j$, and so some $x_i \notin \{ 0,1 \}$.

If $x_i \notin\{ 0,1 \}$ for some $i \geq 2$, then
    \[
    G(\mathbf{x}) - T_j  \geq (n+i) P_m(x_i)  -T_j \geq (n+2)(m-3) - T_j = n-3-j >0.
    \]
Thus $x_i \in \{0,1\}$ for every $i \geq 2$.

Hence at least one of $x_0, x_1$ lies outside $\{ 0, 1\}$.
If either lies outside $\{0,1,-1\}$, then
    \[
    G(\mathbf{x}) \geq nm =n \, \frac{3n+3}{2} > T_j,
    \]
a contradiction.
Therefore at least one of $x_0, x_1$ is equal to $-1$.

Suppose that $x_1=-1$. Then the remaining variables must represent
    \[
    T_j - (n+1) P_m(-1) = \frac{n+2j+3}{2}.
    \]
Since
    \[
    0 < \frac{n+2j+3}{2} < n,
    \]
this is impossible, as every coefficient of the remaining form $\qf{n, n+2, \ldots, 2n}_m$ is at least $n$.
Hence $x_1 \neq -1$, and therefore $x_0=-1$.
The remaining variables must then represent
    \[
    T_j-nP_m(-1)=2n+j.
    \]
Since $x_1, x_2, \ldots, x_n \in \{ 0,1\}$, the integer $2n+j$ would have to be a subsum of $\{ n+1, n+2, \ldots, 2n \}$.
But a single coefficient is at most $2n$, while the sum of any two coefficients is at least $(n+1)+(n+2)=2n+3$.
Thus neither $2n+1$ nor $2n+2$ can occur, a contradiction.

Therefore $G$ represents neither $T_1$ nor $T_2$.
\end{proof}

\smallskip

The three propositions above cover all integers $m \geq \frac{3n+2}{2}$.
Moreover, both $F$ and $G$ represent every integer from $n$ to $2n$ and no positive integer less than $n$.
Hence the preceding propositions provide, for every $m \geq \frac{3n+2}{2}$, an $m$-gonal form which represents all integers from $n$ to $2n$ but is not tight
$\calT(n)$-universal. Therefore the upper bound $m \leq \left\lfloor \frac{3n+1}{2} \right\rfloor$ for the universality range is optimal.

\smallskip

\section{Completion of the Proof of the Main Theorem}

\begin{proof}[Proof of Theorem~\ref{Thm:MainTheorem}]
For the pairs $(m,n)$ indicated by the green and yellow dots, Theorems~\ref{Thm:04} and~\ref{Thm:05}, together with
\cite[Proposition 3.5]{jJ-mK-2023}, yield
    \[
    \mathrm{CS}(m,n)= \{ n, n+1, \ldots, 2n \}.
    \]
The case $m=5$ follows from \cite[Proposition 3.3]{mKim-2022}, \cite[Table 1]{jJ-mK-2023}, while the case $m=6$ reduces to the case $m=3$.
The remaining pairs in the range of the theorem are covered by the previously known results in \cite{mKim-2022, mKim-bkO-2023} and \cite[Theorem 4.7 and Table 2]{jJ-mK-2024}, as indicated by the black dots in Figure~\ref{fig:CS(m,n)}.
Finally, Section~\ref{Sec:Optimality} shows that the upper bound
    \[
    m \leq \left\lfloor \tfrac{3n+1}{2} \right\rfloor
    \]
in the criterion-set formula of the Main Theorem is best possible.
This completes the proof.
\end{proof}

\smallskip


\end{document}